\documentclass[11pt,paper]{amsart}
\usepackage[usenames]{color}
\usepackage{amsmath,pdfsync,verbatim,graphicx,epstopdf,enumerate,fancybox}
\usepackage[colorlinks=true]{hyperref}
\hypersetup{urlcolor=black, citecolor=red, linkcolor=red}
\usepackage{url}
\usepackage[notref,notcite]{}
\usepackage[framemethod=tikz]{mdframed}
\newtheorem{theorem}{Theorem}[section]
\newtheorem{lemma}[theorem]{Lemma}

\newtheorem{definition}[theorem]{Definition}

\numberwithin{equation}{section}

\newcommand{\PD}{\partial}

\newcommand{\Ic}{\mathcal{I}}
\newcommand{\Jc}{\mathcal{J}}

\newcommand{\Sc}{\mathcal{S}}

\newcommand{\Rb}{\mathbb{R}}
\newcommand{\Sb}{\mathbb{S}}

\newcommand{\Beq}{\begin{equation}}
\newcommand{\Eeq}{\end{equation}}
\newcommand{\beq}{\begin{equation*}}
\newcommand{\eeq}{\end{equation*}}
\newcommand{\bal}{\begin{align}}
\newcommand{\eal}{\end{align}}

\renewcommand{\l}{\langle}
\renewcommand{\r}{\rangle}
\newcommand{\D}{\mathrm{d}}
\newcommand{\bp}{\begin{prob}}
	\newcommand{\ep}{\end{prob}}
\newcommand{\bpr}{\begin{proof}}
	\newcommand{\epr}{\end{proof}}

\usepackage{amsmath,tikz}
\usetikzlibrary{matrix}

\newcommand*\Rn{\mathbb{R}^n}

\newcommand{\tn}{T\mathbb{S}^{n-1}}
\renewcommand{\d}{\mathrm{d}}

\usepackage[centerlast,small,sc]{caption}
\title{The generalized Saint Venant operator and integral moment transforms}

\author[Rohit Kumar Mishra and Suman Kumar Sahoo]{Rohit Kumar Mishra$^\ast$ and Suman Kumar Sahoo$^\dagger$}

	\email{rohit.m@iitgn.ac.in, rohittifr2011@gmail.com, suman.k.sahoo@jyu.fi}
	\address{$^\ast$ Indian Institute of Technology Gandhinagar, Gujarat, India
		\newline\indent$^\dagger$ University of Jyv\"{a}skyl\"{a}, Finland}

\begin{document}
	\maketitle
\begin{abstract}
In this article, we work with a generalized Saint Venant operator introduced by Vladimir Sharafutdinov \cite{Sharafutdinov1994} to describe the kernel of the integral moment transforms over symmetric $m$-tensor fields in $n$-dimensional Euclidean space.  We also provide an equivalence between the injectivity question for the integral moment transforms and generalized Saint Venant operator  over  symmetric tensor fields of Schwartz class.
\end{abstract}
	\section{Introduction}\label{sec:introduction}

The space of covariant symmetric $m$-tensor fields on $\Rb^n$ with components in the Schwartz space 
 will be denoted by  $\Sc(S^m)$. 
 In Cartesian coordinates, an element $f \in \mathcal{S}(S^m)$ can be written as
$$ f(x) = f_{i_1\dots i_m}(x)\, dx^{i_1} \cdots dx^{i_m}$$
where  $f_{i_1 \dots i_m} \in \Sc(\Rb^n)$ are symmetric in all indices. For repeated indices, Einstein summation convention will be assumed throughout this article. Moreover, we will not distinguish between covariant tensors and contravariant tensors since we work with the Euclidean metric. 

The \textit{Saint Venant operator} $W: C^\infty(S^m)\rightarrow  C^\infty(S^m \otimes S^m)$ is defined by 
\begin{align}\label{def:Saint Venant operator}
	(W f)_{i_1\dots i_m j_1\dots j_m} &= \sigma(i_1 \dots i_m)\,\,\sigma(j_1 \dots j_m) \sum_{\ell =0}^{m} (-1)^\ell\, \begin{pmatrix}
	m\\
	\ell
	\end{pmatrix}\, \frac{\partial^{m} f_{i_1\dots i_{m-\ell}j_1\dots j_\ell}}{\partial x^{j_{\ell+1}}\dots \partial x^{j_{m}}\partial x^{i_{m-\ell+1}}\dots \partial x^{i_{m}}},
\end{align}
where $\sigma$ is the symmetrization operator defined below (please see equation \eqref{eq:definition of sigma}). This operator $W$ was named after the French mathematician Barr\'e de Saint Venant. In one dimension, the Saint Venant operator describes the unsteady water flow and simplifies the shallow water equations.  This operator appears in various fields such as deformation theory, elasticity, and many more (see \cite{Sombuddha_elastodynamics} and the references therein).

For a vector field $f$ ($m=1$) the equation $Wf = 0$ gives the well known integrability condition $ \frac{\partial f_i}{\partial x^j} - \frac{\partial f_j}{\partial x^i}=0$ for Pfaff form. For a symmetric 2-tensor field $f$ the condition $Wf = 0$ reduces to 
$ \frac{\partial f_{ij}}{\partial^2 x^k \partial x^\ell}+\frac{\partial f_{k\ell}}{\partial^2 x^i \partial x^j}-\frac{\partial f_{i\ell}}{\partial^2 x^k \partial x^j}-\frac{\partial f_{kj}}{\partial^2 x^i \partial x^\ell}=0$
which was derived by Saint Venant and usually called as the deformations compatibility condition. This paper aims to describe the kernel of the integral moment transforms using a  generalized version of the Saint Venant operator on the Schwartz class of tensor fields.

For a non-negative integer $q \geq 0$, the $q$-th integral moment transform of a symmetric $m$-tensor field is the function $I^q :{\Sc}(S^m)\rightarrow{\Sc}(\tn)$ given by  \cite{Sharafutdinov_Generalized_Tensor_Fields,Sharafutdinov1994}:
\begin{equation}\label{eq:definition of momentum ray transform}
(I^q f)(x,\xi)=\int\limits_{-\infty}^\infty t^q\langle f(x+t\xi),\xi^m\rangle dt = \int\limits_{-\infty}^\infty t^q f_{i_1\dots i_m}(x+t\xi)\,\xi^{i_1} \cdots \xi^{i_m} dt,
\end{equation}
where $T{\Sb}^{n-1}=\{(x,\xi)\in{\Rb}^n\times{\Rb}^n\mid |\xi|=1,\langle x,\xi\rangle=0\}$ denotes the space of oriented lines in $\Rb^n$. These transforms were introduced by Sharafutdinov  
and have been investigated by many authors (see for instance \cite{Anuj_Rohit,Krishnan2018,Krishnan2019a,Rohit_Suman2020,BKS_21} and the references therein). 


Observe that the right hand side (R.H.S.) of \eqref{eq:definition of momentum ray transform} is valid even for $ (x,\xi) \in \Rb^n \times \Rb^n\setminus\{0\} $. Therefore, we also  define the extended integral moment transforms $ J^q:{\Sc}(S^m)\rightarrow C^{\infty}\left(\Rb^n \times \Rb^n\setminus\{0\}\right) $ by 
\begin{equation}\label{eq:definition of extended momentum ray transform}
(J^q f)(x,\xi)=\int\limits_{-\infty}^\infty t^q\langle f(x+t\xi),\xi^m\rangle dt = \int\limits_{-\infty}^\infty t^q f_{i_1\dots i_m}(x+t\xi)\,\xi^{i_1} \cdots \xi^{i_m} dt.
\end{equation}
For any fixed integer $k\ge 0$,  the data $(I^0 f, I^1 f, \dots , I^k f)$ and $(J^0 f, J^1 f, \dots , J^k f)$ are equivalent, in fact, there is an explicit relation between these operators (see \cite{Krishnan2018}) \begin{equation}\label{eq:relation between Ik and Jk}
(J^q f)(x,\xi)=|\xi|^{m-2q-1}\sum\limits_{\ell=0}^q(-1)^{q-\ell}{q\choose\ell}\,|\xi|^\ell\, \langle\xi,x\rangle^{q-\ell}\,(I^\ell f)
\left(x-\frac{\langle x, \xi \rangle}{|\xi|^2}\xi,\frac{\xi}{|\xi|}\right).
\end{equation}
The operators $ I^qf(x,\xi) $ obey nice decay property in the first variable. On the other hand the operators $ J^qf(x,\xi) $ are smooth with respect to  both variable and the partial derivatives $ \PD_{x^i}, \PD_{\xi^i} $ are well defined on $ J^q f $.

 We denote the collection of  first $(k+1)$ integral moment transforms of $f \in \Sc(S^m)$ by $\Ic^k f$. More specifically, the operator $\Ic^k: \mathcal{S}(S^m)\rightarrow \left(\mathcal{S}(\tn)\right)^{k+1}$ is  defined by
\begin{align}\label{eq:definition of Ik moment transforms}
 \Ic^k(f)(x, \xi) =  \left(I^0 f(x, \xi),I^1 f(x, \xi),\dots, I^k f(x, \xi)\right), \quad \mbox{ for } (x, \xi) \in \tn.   
\end{align} 
The case $ k=0 $, $\Ic^0=I=I^0$,  corresponds to the classical ray transform of symmetric $m$-tensor fields in $\Rb^n$ and it is well known that $I^0$ has a non-trivial kernel consisting of all potential tensor fields.  An equivalent way to describe the kernel of $I^0$ for  compactly supported symmetric $m$ ($m>0$) tensor fields was presented in terms of Saint Venant operator by Sharafutdinov \cite[Theorem 2.2.1]{Sharafutdinov1994}.  
Additionally, the kernel of the operator $\Ic^k$ was also discussed for compactly supported tensor fields \cite[Theorem 2.1.7.2]{Sharafutdinov1994} in terms of generalized Saint Venant operator $ W^k$ (defined in the next section).   In this article,  we aim to study the operator $W^k$ in  detail to give an alternate kernel description (similar to \cite[Theorem 2.1.7.2]{Sharafutdinov1994}) for the operator $ \Ic^k $ on Schwartz class of symmetric $m$-tensor fields (see Theorem \ref{th:equivalent description of kernel} for more details).  The proofs are completely new and based on the ideas developed by authors in their recent article \cite{Rohit_Suman2020}. 



The rest of the article is organized as follows.  In Section \ref{sec:Def and notation}, we define certain differential operators (including $W^k$) that we use throughout the article. Then we state some known results for the integral moment transforms. Section \ref{sec: Main results} contains the main results of this article and their proofs. 
\section{Preliminaries}\label{sec:Def and notation}
In this section, we recall some known facts (including definitions, notations, and lemmas) about the integral moment transforms and Saint Venant operator, which we will be using throughout this article.  A detailed discussion for these facts can be found in \cite{Krishnan2018} and  also in the book \cite[Chapter 2]{Sharafutdinov1994}.
\subsection{Some differential operators}
 Let $T^m = T^m (\mathbb{R}^n$) denotes the space of $m$-tensors on $\Rb^n$. There is a natural projection of $T^m$ onto the space of symmetric tensors $S^m$, $\sigma : T^m \rightarrow S^m$ given by  
\begin{align}\label{eq:definition of sigma}
(\sigma v)_{i_1 \dots i_m}=\sigma(i_1 \dots i_m) v = \frac{1}{m!}\sum_{\pi \in \Pi_m} v_{\pi(i_1)\dots \pi(i_m)},  \quad\mbox{ for } \quad  v \in T^m
\end{align}
where $\Pi_m$ is the set of permutation of order $m$.

Using this symmetrization operator $\sigma$,  we define the
operator of \textit{inner differentiation} or \textit{symmetrized derivative} $\D:C^\infty(S^m)\rightarrow C^\infty(S^{m+1})$ by
	$$(\D u)_{i_1\dots i_mi_{m+1}} = \sigma(i_1, \dots , i_m) \left(\frac{\partial u_{i_1\dots i_m} }{\partial x_{i_{m+1}}}\right), \quad \mbox{ where } \sigma  \mbox{ is defined in } \eqref{eq:definition of sigma}.$$
	
Given a symmetric $m$-tensor field, we define a symmetric $(m -\ell)$-tensor field  $f^{i_1\cdots i_\ell}$ obtained from $ f $ by fixing the first $\ell$ indices $ i_1,\dots,i_\ell$. This can be done by fixing any $\ell$ indices. Due to symmetry it is enough to fix the first $\ell$ indices, that is,
	\begin{equation}\label{def: of restricted tensor field}
f^{i_1\cdots i_\ell}_{j_1\dots j_{m-\ell}}= f_{i_1\dots i_\ell j_1\dots j_{m-\ell}}, \quad \mbox{ where } i_1,\dots, i_{\ell} \  \mbox{ are fixed.}
	\end{equation} 
\noindent Next, we introduce the generalized Saint Venant operator, the primary object of study in this article. 
\begin{definition}[Generalized Saint Venant operator,\cite{Sharafutdinov1994}]\label{def:deneralized Snt _Venant_ope}
	For $m \geq 0$ and $ 0\le k \le m$, the generalized Saint Venant operator (of order $k$) $ W^k:C^\infty(S^m)\rightarrow C^\infty(S^{m-k} \otimes S^m) $ is defined by the equality 
\small{\begin{align}\label{eq:Genralized Saint-Venant operator}
&	(W^k f)_{p_1\dots p_{m-k} q_1\dots q_{m-k} i_1\dots i_k}\nonumber\\ & \qquad \qquad = \sigma(p_1 \dots p_{m-k})\sigma(q_1 \dots q_{m-k} i_1\dots i_k) \sum_{\ell =0}^{m-k} (-1)^\ell \begin{pmatrix}
	m-k\\
	\ell
	\end{pmatrix} 	\frac{\partial^{m-k} f^{i_1\dots i_k}_{p_1\dots p_{m-k-\ell}q_1\dots q_\ell}}{\partial x^{p_{m-k-\ell+1}}\dots \partial x^{p_{m-k}}\partial x^{q_{\ell+1}}\dots \partial x^{q_{m-k}}}.
	\end{align}}
\end{definition}
\noindent Note that, $W^k$ is a differential operator of order $(m-k)$.
\noindent For $k=0$, this is well known Saint-Venant operator $W$ defined above (see equation \eqref{def:Saint Venant operator}). There is an equivalent way to define the Saint Venant operator $W$ which we discuss next. This equivalent formulation will be used   to simplify several calculations. 
\begin{definition}\cite[Chapter 2]{Sharafutdinov1994}
	We define the operator $R : \mathcal{S}(S^m)\rightarrow \mathcal{S}(T^{2m})$ as follows
	\begin{align}\label{def:R}
	(Rf)_{i_1j_1 \dots i_mj_m}= \alpha(i_1j_1)\alpha(i_2j_2) \dots \alpha(i_mj_m) \frac{\PD^m\!f_{i_1\dots i_m}}{\PD x^{j_1}\dots \PD x^{j_m}}
	\end{align}
	where $\alpha(i_1i_2)$ gives \textit{alternation} with respect to two indices, that  is,  $$ \alpha(i_1i_2)g_{i_1i_2\cdots i_m}= \frac{1}{2}\left( g_{i_1i_2\cdots i_m}- g_{i_2i_1\cdots i_m}\right), \qquad \mbox{ for } g \in T^m(\mathbb{R}^n).$$
\end{definition}
\noindent The operators $R$ and $W$ are equivalent in the sense that they satisfy the following two relations \cite[Equations 2.4.6 and 2.4.7] {Sharafutdinov1994}:
\begin{equation}\label{relation_bet_rf_wf}
\begin{aligned}
(Wf)_{i_1\dots i_mj_1\dots j_m}&= \sigma(i_1\dots i_m)\sigma(j_1\dots j_m) \, (Rf)_{i_1j_1 \dots i_mj_m},\\
(Rf)_{i_1j_1 \dots i_mj_m}&= (m+1) \alpha(i_1j_1)\alpha(i_2j_2) \dots \alpha(i_mj_m) \, (Wf)_{i_1\dots i_mj_1\dots j_m}.
\end{aligned}
\end{equation}
\subsection{Some known results for integral moment transforms}
The extended $q$-th integral moment ray transform of the tensor field $f^{i_1\cdots i_\ell}$  for any fixed choice of $i_1,\dots,i_\ell$ will be denoted by $J^q f^{i_1\cdots i_\ell} (x, \xi)$, for any integer $ q \ge 0$.  The following result \cite{Rohit_Suman2020} provides a way to compute the ray transform of $f^{i_1\cdots i_k}$ from the knowledge of $\Ic^k f$ for $ 0 \leq k \leq m$.
\begin{lemma}\cite[Lemma 7]{Rohit_Suman2020}\label{inversion}
The  following identity holds for any $f \in \mathcal{S}(S^m)$:	
	\begin{equation}\label{inversion_general}
	J^0f^{i_1\cdots i_r}= \frac{ (m-r)!}{m!}\sigma(i_1\dots i_r)\sum_{p=0}^{r} (-1)^p \binom{r}{p}\, \frac{\partial^r J^pf}{\partial x^{i_1}\dots\partial x^{i_p}\partial\xi^{i_{p+1}}\dots\partial\xi^{i_r}}, \quad \mbox{ for  } \ \ 1\le i_1,\dots,i_r\le n.
	\end{equation}
\end{lemma}

\begin{lemma}\cite[Lemma 2.6]{Krishnan2019a} \label{L2.5}
	Let a function $\psi\in C^\infty\big({\Rb}^n\times{\Rb}^n\setminus\{0\})\big)$ be positively homogeneous of degree $\lambda$ in the second argument
	\begin{equation}
	\psi(x,t\xi)=t^\lambda\psi(x,\xi)\quad(t>0).
	\label{Eq2.10}
	\end{equation}
 Assume the restriction $\psi\left|_{T{\mathbb S}^{n-1}} \in {\mathcal S}(T{\mathbb S}^{n-1}) \right.$. Further assume the restriction of  $\l\xi,\partial_x\r\psi$  and
 all its derivatives to $T{\mathbb S}^{n-1}$ belong to ${\mathcal S}(T{\mathbb S}^{n-1})$, that is,
	\begin{equation}
	\left.\frac{\partial^{k+\ell}(\l\xi,\partial_x\r\psi)}{\partial x^{i_1}\dots\partial x^{i_k}\partial \xi^{j_1}\dots\partial \xi^{j_\ell}}\right|_{T{\mathbb S}^{n-1}}\in{\mathcal S}(T{\mathbb S}^{n-1})\quad\mbox{for all}\quad 1\le i_1,\dots,i_k,j_1,\dots,j_\ell\le n.
	\label{Eq2.11}
	\end{equation}
	Then the restriction to $T{\mathbb S}^{n-1}$ of every derivative of $\psi$ also belongs to ${\mathcal S}(T{\mathbb S}^{n-1})$, i.e.,
	\begin{equation}
	\left.\frac{\partial^{k+\ell}\psi}{\partial x^{i_1}\dots\partial x^{i_k}\partial \xi^{j_1}\dots\partial \xi^{j_\ell}}\right|_{T{\mathbb S}^{n-1}}
	\in{\mathcal S}(T{\mathbb S}^{n-1})\quad\mbox{for all}\quad 1\le i_1,\dots,i_k,j_1,\dots,j_\ell\le n.
	\label{Eq2.12}
	\end{equation}
\end{lemma}
\section{Main results and their proofs}\label{sec: Main results}
\noindent 
The main result of the article provides a kernel description of the operator $\Ic^k$ in terms of the generalized Saint Venant operator $W^k$. An equivalent kernel description is also presented in terms of potential tensor fields, but this description is a bit restrictive as discussed in the second theorem below. 

\begin{theorem}\label{equi_of_ik and w^k}	
	Let $ f\in \mathcal{S}(S^m) $ in $ \mathbb{R}^n\, (n\ge 2) $ and $ 0\le k\le m $. Then  $ \Ic^k f=0  $ if and only if $ W^k f=0 $. That is,  the operators $\Ic^k$ and $W^k$ have the same kernel.
\end{theorem}
The proof of this theorem is completely new even for $ k=0$. Note that we do not have any restriction on the dimension which arises naturally in \cite[Theorem 6]{Rohit_Suman2020}.  The following result uses the dimension restriction coming from \cite{Rohit_Suman2020} to relate the above result with \cite[Theorem 6]{Rohit_Suman2020}.  This theorem is known for compactly supported symmetric tensor fields in the case $k=0$ \cite[Theorem 2.2.1]{Sharafutdinov1994}.



	\begin{theorem}\label{th:equivalent description of kernel}
		Let $ f\in \mathcal{S}(S^m) $ and $ k $ be an integer such that $ 1\le k\le \min\{m,n-1\}  $. Then the following conditions are equivalent:
		\begin{itemize}
			\item[(1)] $ \Ic^kf=0.$ 
			\item[(2)] $ f =  \d^{k+1} v$, for some  $(m-k-1)$-tensor field $v$ satisfying  $\d^{\ell}v \rightarrow 0 $ as $ |x| \rightarrow \infty$ for $ 0\le \ell\le k $.
			\item[(3)] $ W^kf=0 $.
 		\end{itemize} 
	\end{theorem}
	If we assume that Theorem \ref{equi_of_ik and w^k} holds, then  this theorem's proof  follows from the known chain of equivalence relations given below.  
	\begin{proof}
		$ (1) \Leftrightarrow (2) $ follows from \cite[Theorem 6]{Rohit_Suman2020}, $(1) \Leftrightarrow (3)  $ follows from Theorem \ref{equi_of_ik and w^k} and  $ (2) \Leftrightarrow (3)  $ holds trivially. 
	\end{proof}	
The remainder of the article focuses on the proof of  Theorem \ref{equi_of_ik and w^k}. The proof of this theorem is divided into several lemmas.  
\begin{lemma}[\cite{Rohit_Suman2020}]\label{th:generalized saint vanent operator}
		Let $ f\in \mathcal{S}(S^m) $ and $ 0\le k\le m $. The generalized Saint Venant operator $ W^k\! f$  can be recovered explicitly from the knowledge of $\Ic^k f$. 
	\end{lemma}
	The authors proved this lemma in \cite{Rohit_Suman2020}. 
However, we prefer to sketch the proof here because some intermediate steps are essential for upcoming lemmas.To prove this lemma, we need to recall an important second order differential operator known as the John operator from \cite[Theorem 2.10,1]{Sharafutdinov1994}. The John operator is denoted by $ \Jc_{pq}: C^{\infty}(\Rn \times \Rn \setminus\{0\})\rightarrow C^{\infty}(\Rn \times \Rn \setminus\{0\})  ,\, (1 \leq p, q \leq n)$ and given by
\begin{equation}
    \Jc_{pq} = \frac{\partial^2}{\partial x^p \partial \xi^q}-\frac{\partial^2}{\partial x^q \partial \xi^p}.
\end{equation}
Please note $\Jc$ denotes the  John operator while $J$ is used for extended integral moment transform.
\begin{proof}
It is sufficient to prove that $W^k f$ can be determined from $J^0f, \dots, J^k f$ because knowing $\Ic^k f = (I^0 f, \dots , I^k f)$ is equivalent to knowing $J^0f, \dots, J^k f$.  Now, for fixed $1 \leq i_1, \dots, i_k\leq n$,  the following is known from Lemma \ref{inversion} 
	\begin{align}\label{Eq57}
	J^0f^{i_1\cdots i_k}= \frac{ (m-k)!}{m!}\sigma(i_1\dots i_k)\sum_{p=0}^{k} (-1)^p \binom{k}{p}\, \frac{\partial^k J^pf}{\partial x^{i_1}\dots\partial x^{i_p}\partial\xi^{i_{p+1}}\dots\partial\xi^{i_k}}.
	\end{align}
Applying $\Jc$ to $	J^0f^{i_1\cdots i_k}$, we obtain
	\begin{align*}
	\left(\Jc( J^0f^{i_1\cdots i_k})\right)_{p_1q_1}=2\,(m-k)\alpha(p_1q_1)\int_{\mathbb{R}} \xi^{j_1}\dots \xi^{j_{m-k-1}} \frac{\partial f^{i_1\cdots i_k}_{j_1\dots j_{m-k-1}p_1}}{\partial x^{q_1}} (x+t\xi) dt.
	\end{align*}
	Applying the John operator $(m-k-1)$ more times to the above equation and repeating the same arguments, we obtain 
	\begin{align}
	\left(\Jc^{m-k} \left(J^0f^{i_1\cdots i_k}\right)\right)_{p_1q_1\dots p_{m-k}q_{m-k}} 
	&=2^{m-k}(m-k)!\int_{\mathbb{R}} (Rf^{i_1\dots i_k})_{p_1q_1\dots p_{m-k}q_{m-k}} (x+t \xi) dt. \label{eq: Saint-Venant}
	\end{align}
The right hand side is the ray transform of  scalar function $(Rf^{i_1\dots i_k})_{p_1q_1\dots p_{m-k}q_{m-k}} $ for all possible choices of indices $1 \leq p_1, q_1, \dots , p_{m-k}, q_{m-k} \leq n$.  Thus $R f^{i_1\dots i_k}$ can be determined explicitly by inverting X-ray transform of scalar functions \cite[Theorem $2.12.2$ for $m=0$]{Sharafutdinov1994}. Knowing $ R f^{i_1\dots i_k} $ is same as knowing  $W f^{i_1\dots i_k}$ from the first relation of  \eqref{relation_bet_rf_wf}.  Finally, to complete the proof of this lemma we need to connect $W f^{i_1\dots i_k}$ and $W^k\,f $. To this end, let us write  $W f^{i_1\dots i_k}$ explicitly
\begin{align*}
&	(W f^{i_1\dots i_k})_{p_1\dots p_{m-k} q_1\dots q_{m-k}}
	\nonumber\\& \qquad  = \sigma(p_1\dots p_{m-k})\,\sigma( q_1\dots q_{m-k})
	\sum_{\ell =0}^{m-k} (-1)^\ell \begin{pmatrix}
	m-k\\
	\ell
	\end{pmatrix}  \frac{\partial^{m-k} f^{i_1\dots i_k}_{p_1\dots p_{m-k-\ell}q_1\dots q_\ell}}{\partial x^{p_{m-k-\ell+1}}\dots \partial x^{p_{m-k}}\partial x^{q_{\ell+1}}\dots \partial x^{q_{m-k}}}.
	\end{align*}
	Here, we make the following  observation 
	\begin{align}\label{Eq60}
	(W^k f)_{p_1\dots p_{m-k} q_1\dots q_{m-k} i_1\dots i_k}  &= \sigma(q_1, \dots, q_{m-k}, i_1,\dots, i_k)(Wf^{i_1\dots i_k})_{p_1\dots p_{m-k}q_1\cdots q_{m-k}}.
	\end{align}
	Now,  right-hand side of \eqref{Eq60} is completely known to us in terms of $ J^0f,\dots, J^k f $ as we discussed above. Therefore we know $W^k f$, which completes the proof.
\end{proof}

\begin{lemma}[Main Lemma]\label{Lm_1}
Let $ 0\le k <m $ and  $ W^kf=0 $ for some $ f \in \mathcal{S}(S^m) $, then we have
\begin{align}\label{m-k_derivatiove_of_J^0f_r}
\sigma(q_1, \dots, q_{m-k}, i_1,\dots, i_k)\Bigg[\frac{\PD^{m-k}}{\PD x^{q_{m-k}}\cdots \PD x^{q_1}}(J^0f^{i_1\cdots i_k})\Bigg]=0.
\end{align}
\end{lemma}
\begin{proof}
  Assume that $ W^kf=0$ then from \eqref{Eq60} we have 
 \begin{align*}
  \sigma(q_1, \dots, q_{m-k}, i_1,\dots, i_k)	(Wf^{i_1\dots i_k})=0 \quad \mbox{for fixed} \quad 1\le i_1\cdots i_k\le n. 
 \end{align*} 
  This together with \eqref{eq: Saint-Venant} and the second equation of \eqref{relation_bet_rf_wf} entails
 \begin{align}\label{johns_condition_j0f_k}
 \sigma(q_1, \dots, q_{m-k}, i_1,\dots, i_k)\sigma(p_1,\cdots,p_{m-k})	\left(\Jc^{m-k} (J^0f^{i_1\cdots i_k})\right)_{p_1q_1\dots p_{m-k}q_{m-k}}=0.
 \end{align}

Since the symmetrization operators $ \sigma(q_1, \dots, q_{m-k}, i_1,\dots, i_k)$ and $ \sigma(p_1,\cdots,p_{m-k}) $ commute with each other, this together with \eqref{johns_condition_j0f_k} implies
 \begin{align*}
 \sigma(p_1,\cdots,p_{m-k})   \sigma(q_1, \dots, q_{m-k}, i_1,\dots, i_k)	\left(\Jc^{m-k} (J^0f^{i_1\cdots i_k})\right)_{p_1q_1\dots p_{m-k}q_{m-k}}=0.
 \end{align*}
  Now multiplying above by a symmetric $m-k$ tensor  $  \xi^{p_{m-k}}\cdots \xi^{p_1}$ we obtain
 \begin{align*}
   \xi^{p_{m-k}}\cdots \xi^{p_1}  \sigma(p_1,\cdots,p_{m-k})   \sigma(q_1, \dots, q_{m-k}, i_1,\dots, i_k)	\left(\Jc^{m-k} (J^0f^{i_1\cdots i_k})\right)_{p_1q_1\dots p_{m-k}q_{m-k}}&=0\\
   \sigma(p_1,\cdots,p_{m-k}) \left( \xi^{p_{m-k}}\cdots \xi^{p_1}     \sigma(q_1, \dots, q_{m-k}, i_1,\dots, i_k)	\left(\Jc^{m-k} (J^0f^{i_1\cdots i_k})\right)_{p_1q_1\dots p_{m-k}q_{m-k}}\right)&=0.
 \end{align*}
Taking summation over $ p_1,\cdots p_{m-k}$ we get 
 \begin{align}\label{main_relation}
 	\sigma(q_1, \dots, q_{m-k}, i_1,\dots, i_k) \, \xi^{p_{m-k}}\cdots \xi^{p_1} \,	\left(\Jc^{m-k} (J^0f^{i_1\cdots i_k})\right)_{p_1q_1\dots p_{m-k}q_{m-k}}=0.
 \end{align}
Since $ (J^0f^{i_1\cdots i_k})  $ is the ray transform of a symmetric $ m-k $ tensor field, from the definition we have \[ (J^0f^{i_1\cdots i_k})(x+t\xi,\xi) = (J^0f^{i_1\cdots i_k})(x,\xi),\quad (J^0f^{i_1\cdots i_k})(x,\lambda\xi)=\lambda^{m-k-1}(J^0f^{i_1\cdots i_k})(x,\xi)  \]
for $ \lambda>0 $ and $ t\in \mathbb{R}$. This immediately gives
\begin{equation}\label{translation_and_homogeneity_relation}
\begin{aligned}
	\l \xi,\PD_x\r \left((J^0f^{i_1\cdots i_k})\right)&= 0,\\
\l \xi,\PD_{\xi}\r \Jc^{\ell} \left((J^0f^{i_1\cdots i_k})\right)&=(m-k-1-\ell)\, \Jc^{\ell} \left((J^0f^{i_1\cdots i_k})\right) \quad \mbox{ for } \quad 0\le \ell \le (m-k-1).
\end{aligned}
\end{equation}
Now we compute L.H.S of \eqref{main_relation} without symmetrization.  
\begin{equation}\label{Eq19}
\begin{aligned}
  &\xi^{p_1}\cdots \xi^{p_{m-k}} \,	\left(\Jc^{m-k} (J^0f^{i_1\cdots i_k})\right)_{p_1q_1\dots p_{m-k}q_{m-k}}\\&=\xi^{p_1}\cdots \xi^{p_{m-k-1}}\,\bigg[	\PD_{\xi^{q_{m-k}}} \l \xi,\PD_{x}\r 	\left(\Jc^{m-k-1} (J^0f^{i_1\cdots i_k})\right)_{p_1q_1\dots p_{m-k-1}q_{m-k-1}}  \\&\qquad - \PD_{x^{q_{m-k}}}\left(\Jc^{m-k-1} (J^0f^{i_1\cdots i_k})\right)_{p_1q_1\dots p_{m-k-1}q_{m-k-1}} \\& \qquad \quad - \PD_{x^{q_{m-k}}}\l \xi,\PD_{\xi}\r  \left(\Jc^{m-k-1} (J^0f^{i_1\cdots i_k})\right)_{p_1q_1\dots p_{m-k-1}q_{m-k-1}}\bigg].
 \end{aligned}
 	\end{equation}
 Using \eqref{translation_and_homogeneity_relation} and the fact that John operator commutes with $ \l \xi,\PD_{x}\r $, we get
\begin{align*}
	 \l \xi,\PD_{x}\r 	\left(\Jc^{m-k-1} (J^0f^{i_1\cdots i_k})\right)=&0 \quad \mbox{and} \quad  \l \xi,\PD_{\xi}\r  \left(\Jc^{m-k-1} (J^0f^{i_1\cdots i_k})\right)=0.
\end{align*}
 This together with \eqref{Eq19} gives
 \begin{equation}\label{main_relation_1}
 \begin{aligned}
 &\xi^{p_1}\cdots \xi^{p_{m-k}} \,	\left(\Jc^{m-k} (J^0f^{i_1\cdots i_k})\right)_{p_1q_1\dots p_{m-k}q_{m-k}}\\
 &\qquad \qquad =(-1)\,  \PD_{x^{q_{m-k}}}\bigg[ \xi^{p_1}\cdots \xi^{p_{m-k-1}}\,\left(\Jc^{m-k-1} (J^0f^{i_1\cdots i_k})\right)_{p_1q_1\dots p_{m-k-1}q_{m-k-1}} \bigg].	
 \end{aligned}
 \end{equation}
Repeating the same analysis as in \eqref{Eq19} we get
	\begin{align*}
&\xi^{p_1}\cdots \xi^{p_{m-k-1}} \,	\left(\Jc^{m-k-1} (J^0f^{i_1\cdots i_k})\right)_{p_1q_1\dots p_{m-k-1}q_{m-k-1}}\\&\quad =\xi^{p_1}\cdots \xi^{p_{m-k-2}}\,\bigg[	\PD_{\xi^{q_{m-k-1}}} \l \xi,\PD_{x}\r 	\left(\Jc^{m-k-2} (J^0f^{i_1\cdots i_k})\right)_{p_1q_1\dots p_{m-k-2}q_{m-k-2}}\\&\qquad - \PD_{x^{q_{m-k-1}}}\left(\Jc^{m-k-2} (J^0f^{i_1\cdots i_k})\right)_{p_1q_1\dots p_{m-k-2}q_{m-k-2}}\\
& \qquad\qquad   - \PD_{x^{q_{m-k-1}}}\l \xi,\PD_{\xi}\r \left(\Jc^{m-k-2} (J^0f^{i_1\cdots i_k})\right)_{p_1q_1\dots p_{m-k-2}q_{m-k-2}} \bigg].	
	\end{align*}
This, \eqref{main_relation_1}  and together with \eqref{translation_and_homogeneity_relation} gives
\begin{align*}
&\xi^{p_1}\cdots \xi^{p_{m-k}} \,	\left(\Jc^{m-k} (J^0f^{i_1\cdots i_k})\right)_{p_1q_1\dots p_{m-k}q_{m-k}}\\
&\qquad\qquad = (-1)(-2) \PD^2_{x^{q_{m-k}}x^{q_{m-k-1}}}\bigg[ \xi^{p_1}\cdots \xi^{p_{m-k-2}}\, \left(\Jc^{m-k-2} (J^0f^{i_1\cdots i_k})\right)_{p_1q_1\dots p_{m-k-2}q_{m-k-2}} \bigg].	
\end{align*}
 Iterating this $(m-k-2)$ times more, we obtain
 \begin{align}\label{main_relation_4}
\xi^{p_1}\cdots \xi^{p_{m-k}} \,	\left(\Jc^{m-k} (J^0f^{i_1\cdots i_k})\right)_{p_1q_1\dots p_{m-k}q_{m-k}}=& (-1)^{m-k}\, (m-k)!\,\frac{\PD^{m-k}}{\PD x^{q_{m-k}}\cdots \PD x^{q_1}}(J^0f^{i_1\cdots i_k}).
 \end{align}
 Combining \eqref{main_relation} with \eqref{main_relation_4}
 \begin{align*}
 \sigma(q_1, \dots, q_{m-k}, i_1,\dots, i_k)\Bigg[\frac{\PD^{m-k}}{\PD x^{q_{m-k}}\cdots \PD x^{q_1}}(J^0f^{i_1\cdots i_k})\Bigg]=0.
 \end{align*}
  This finishes the proof.
\end{proof}
\begin{lemma}\label{additional_prop}
Suppose the relation \eqref{m-k_derivatiove_of_J^0f_r} holds, then we have 
\begin{align}\label{Final_identity}
\sigma(q_1, \dots, q_{m-r}, i_1,\dots, i_r)\Bigg[\frac{\PD^{m-r}}{\PD x^{q_{m-r}}\cdots \PD x^{q_1}}(J^0f^{i_1\cdots i_r})\Bigg]=0 \quad \mbox{ for} \quad 0\le r \le k.
\end{align}
\end{lemma}
\begin{proof}
For  $ f \in\mathcal{S}(\mathbb{R}^n) $, one can obtain the following  relation for $ 0\le r \le k $ by a direct computation:	
	\begin{align}\label{eq_3.14}
	J^0f^{i_1\dots i_r} (x, \xi)&= \int_{-\infty}^{\infty}  (f^{i_1\dots i_r})_{j_1\dots j_{m-r}}(x+t\xi)\,\xi^{j_1}\cdots \xi^{j_{k-r}} \xi^{j_{k-r+1}}\cdots\xi^{j_{m-r}}\nonumber\\ &= \xi^{i_{1}}\dots\xi^{i_{k-r}}\left(J^0f^{i_1\cdots i_k}\right)(x, \xi).
	\end{align}
	By \cite[Lemma 2.4.1]{Sharafutdinov1994}, for any $m$-tensor $f$, which has symmetry in the first $(m-k)$ indices and last $k$ indices, the following symmetrization relation holds:
	\begin{align}\label{eq_3.15}
	&	\sigma(q_1, \dots, q_{m-k}, i_1,\dots, i_{k}) f_{q_1 \dots q_{m-k} i_1\dots i_k}\nonumber \\&\qquad \qquad = \frac{1}{m} 	\sigma(q_1, \dots, q_{m-k}, i_1,\dots, i_{k-1}) \left(  k f_{q_1 \dots q_{m-k} i_1 \dots i_k}+ (m-k)   f_{i_k q_1 \dots q_{m-k-1} q_{m-k} i_1 \dots i_{k-1}}   \right).
	\end{align}
Multiplying \eqref{m-k_derivatiove_of_J^0f_r} by $ \xi^{i_{1}} $ and then summing over $ i_1 $ we get 
	\begin{align*}
 \xi^{i_1}	\sigma(q_1, \dots, q_{m-k}, i_1,\dots, i_k)\Bigg[\frac{\PD^{m-k}}{\PD x^{q_{m-k}}\cdots \PD x^{q_1}}(J^0f^{i_1\cdots i_k})  \Bigg] = 0.
	 \end{align*}
	Using \eqref{eq_3.15} from above we obtain 
\begin{align*}
 \xi^{i_1} 	  \sigma(q_1, \dots, q_{m-k}, i_2,\dots, i_{k} )\Bigg[k\,\frac{\PD^{m-k}}{\PD x^{q_{m-k}}\cdots \PD x^{q_1}}(J^0f^{i_1\cdots i_k}) +(m-k)  \frac{\PD^{m-k}}{\PD x^{i_1} x^{q_{m-k-1}}\cdots \PD x^{q_1}}(J^0f^{q_{m-k} i_2 \cdots i_k}) \Bigg] =0. 
	\end{align*}
	Since the symmetrization operator $ \sigma(q_1, \dots, q_{m-k}, i_2,\dots, i_{k} )$ is independent of $i_1$, this gives
	\begin{align*}
	 \sigma(q_1, \dots, q_{m-k}, i_2,\dots, i_{k} )\Bigg[k\,\frac{\PD^{m-k}}{\PD x^{q_{m-k}}\cdots \PD x^{q_1}}(J^0f^{i_1\cdots i_k}) \xi^{i_1} +(m-k) \langle\xi,\PD_x\rangle  \frac{\PD^{m-k-1}}{ x^{q_{m-k-1}}\cdots \PD x^{q_1}}(J^0f^{ q_{m-k} i_2 \cdots i_k}) \Bigg] =0.     
	\end{align*}
	Using the first relation in \eqref{translation_and_homogeneity_relation} and the fact that $ \langle\xi,\PD_x\rangle$ commutes with constant coefficient differential operator, we obtain
	\begin{align}\label{eq_3.16}
	    \sigma(q_1, \dots, q_{m-k}, i_2,\dots, i_{k} )\Bigg[\frac{\PD^{m-k}}{\PD x^{q_{m-k}}\cdots \PD x^{q_1}}(J^0f^{i_1\cdots i_k}) \xi^{i_1}\Bigg]=0.
    	\end{align}
	 Multiplying  \eqref{eq_3.16}  by $ \xi^{i_{2}}\dots\xi^{i_{k-r}} $ and summing over the indices $ i_2,\dots, i_{k-r} $ and repeating similar analysis as above  we get
	 \begin{align}\label{Eq33}
	\sigma(q_1, \dots, q_{m-k}, i_{k-r+1},\dots, i_k)\Bigg[\frac{\PD^{m-k}}{\PD x^{q_{m-k}}\cdots \PD x^{q_1}}(J^0f^{i_1\cdots i_k}) \xi^{i_{1}}\cdots\xi^{i_{k-r}} \Bigg] = 0.
	 \end{align}
	 After a re-indexing, combining \eqref{Eq33} with \eqref{eq_3.14} we get
\begin{align*}
\sigma(q_1, \dots, q_{m-k}, i_1,\dots, i_r)\Bigg[\frac{\PD^{m-k}}{\PD x^{q_{m-k}}\cdots \PD x^{q_1}}(J^0f^{i_1\cdots i_r}) \Bigg]=0.
\end{align*}	
Finally differentiating  this  equation with respect to $ x^{q_{m-k+1}} \dots x^{q_{m-r}} $ and then taking $ \sigma(q_1, \dots, q_{m-r}, i_1,\dots, i_r)$, we get
\begin{equation*}
 \sigma(q_1, \dots, q_{m-r}, i_1,\dots, i_r)\Bigg[\frac{\PD^{m-r}}{\PD x^{q_{m-r}}\cdots \PD x^{q_1}}(J^0f^{i_1\cdots i_r})\Bigg]=0 \qquad \mbox{ for} \quad 0\le r \le k.
\end{equation*}
\end{proof}
\begin{proof}[Proof of Theorem \ref{equi_of_ik and w^k}]
The aim is to prove $W^k f = 0$ if and only if $\Ic^k f =0$. We only need to prove the ``only if" part of the statement since the other direction \[\Ic^k f  = 0 \implies W^k f = 0,\] follows from Lemma \ref{th:generalized saint vanent operator}.
\par  In order to prove the ``only if" part, we assume $W^k f= 0$. The idea here is to use Lemma \ref{additional_prop} repeatedly. 
\\
For $ f\in \mathcal{S}(S^m)$, a direct application of integration by parts implies
\begin{align}\label{translation}
\langle \xi, \PD_x  \rangle J^kf =-k\,J^{k-1}f.
\end{align}
\noindent As a first step, we put $r=0$ in \eqref{Final_identity} to get
\[  \sigma(q_1, \dots, q_{m})\Bigg[\frac{\PD^{m}}{\PD x^{q_{m}}\cdots \PD x^{q_1}}(J^0f)\Bigg]=\frac{\PD^{m}}{\PD x^{q_{m}}\cdots \PD x^{q_1}}(J^0f) =0. \]
We know that $J^0f|_{\tn}\in \mathcal{S}(\tn) $ and  $ \langle \xi,\PD_x \rangle J^0f=0$ follows  from substituting $k=0$ in \eqref{translation}. Thus $ J^0f$ satisfies all the hypotheses of Lemma \ref{L2.5}. This implies \[ \frac{\PD^{m-1}}{\PD x^{q_{m-1}}\cdots \PD x^{q_1}}(J^0f)\big|_{\tn} \in \mathcal{S}(\tn) .\] This together with $ \PD_{x^{q_m}}\left( \frac{\PD^{m-1}}{\PD x^{q_{m-1}}\cdots \PD x^{q_1}}(J^0f)\right)=0$, entails $\frac{\PD^{m-1}}{\PD x^{q_{m-1}}\cdots \PD x^{q_1}}(J^0f)=0 $. This can be proved directly (see also  \cite[Statement 2.12]{Krishnan2019a}).  Proceeding in this way after finitely many steps we conclude \begin{align}\label{j_0_f}
    J^0f(x,\xi)=0.
\end{align}
\noindent Next, consider the Lemma \ref{additional_prop} with $ r=1 $ to get \[\sigma(q_1, \dots, q_{m-1}, i_1)\Bigg[\frac{\PD^{m-1}}{\PD x^{q_{m-1}}\cdots \PD x^{q_1}}(J^0f^{i_1})\Bigg]=0. \]
Using  $J^0f^{i_1}= \frac{\PD J^0f}{\PD \xi^{i_1}}-\frac{\PD J^1f}{\PD x^{i_1}}$ (see Lemma \ref{inversion}) together with the fact $ J^0f=0$ in the above equation which gives
\begin{align*}
\sigma(q_1, \dots, q_{m-1}, i_1)\Bigg[\frac{\PD^{m-1}}{\PD x^{q_{m-1}}\cdots \PD x^{q_1}}(J^0f^{i_1})\Bigg]=\frac{\PD^{m}}{\PD x^{q_{m-1}}\cdots \PD x^{q_1}\PD x^{i_1}}(J^1f)=0.	
\end{align*}
 We now apply the similar argument on $J^1f$ as it satisfies following:\[ J^1f|_{\tn}\in \mathcal{S}(\tn) \quad \mbox{and}\quad \langle \xi,\PD_x\rangle J^1f = -J^0f (=0), \quad \mbox{by} \quad \eqref{translation}\quad \mbox{and} \quad \eqref{j_0_f}. \]   Repeating similar argument used above we can conclude $ J^1f=0 $. 
Following the same idea, assume \[ J^pf =0\quad \mbox{for} \quad p=0,1,\cdots, r-1\]  and apply Lemma \ref{inversion} again to get 
\[ J^0f^{i_1\cdots i_r}= \frac{(-1)^r(m-r)!}{m!}  \frac{\partial^r J^rf}{\partial x^{i_1}\dots\partial x^{i_r}}.\]
This together with \eqref{Final_identity} gives
\begin{align*}
 &\sigma(q_1, \dots, q_{m-r}, i_1,\dots, i_r)\Bigg[\frac{\PD^{m-r}}{\PD x^{q_{m-r}}\cdots \PD x^{q_1}}(J^0f^{i_1\cdots i_r})\Bigg]\\
 &\qquad\qquad\qquad= \frac{\PD^{m}}{\PD x^{q_{m-r}}\cdots \PD x^{q_1} \PD x^{i_1}\cdots \PD x^{i_r}}(J^rf)=0, \quad  \mbox{for} \quad 0 \leq r \leq k.
\end{align*}
\noindent Repeating similar analysis this implies $ J^rf=0 $. Therefore $ I^rf=0  $ for $ 0\le r\le k $ or equivalently, $ \Ic^kf=0 $. 
Thus we have proved that \[ W^kf=0 \implies\Ic^kf=0.\]
This completes the proof of our main theorem.
\end{proof}
\textbf{Acknowledgement.} The authors would like to thank Venky P.  Krishnan for several fruitful discussions. 
\bibliographystyle{alpha}

\end{document}